\documentclass[12pt]{article}
\input epsf.tex
\usepackage{graphicx}
\usepackage{amsmath,amsthm,amsfonts,amscd,amssymb,comment,eucal,latexsym,mathrsfs}
\usepackage{stmaryrd}
\usepackage[all]{xy}

\usepackage{epsfig}

\usepackage[all]{xy}
\xyoption{poly}
\usepackage{fancyhdr}
\usepackage{wrapfig}
\usepackage{epsfig}

\usepackage[english]{babel}
\usepackage{enumerate}
\usepackage{xspace}
\theoremstyle{plain}
\newtheorem{thm}{Theorem}[section]
\newtheorem{prop}[thm]{Proposition}
\newtheorem{lem}[thm]{Lemma}
\newtheorem{cor}[thm]{Corollary}
\newtheorem{conj}{Conjecture}

\theoremstyle{definition}
\newtheorem{defn}{Definition}
\theoremstyle{remark}
\newtheorem{remark}{Remark}
\newtheorem{example}{Example}

\advance \topmargin by -\headheight
\advance \topmargin by -\headsep
\evensidemargin \oddsidemargin
     \def\F{{\mathbb{F}}}        \def\N{{\mathbb{N}}}   \def\Q{{\mathbb{Q}}}         \def\Z{{\mathbb{Z}}}

  \def\beta{{\bar{\eta}}}

      \def\cG{{\mathcal{G}}} \def\cH{{\mathcal{H}}}   \def\cK{{\mathcal{K}}}   \def\cN{{\mathcal{N}}}  \def\cP{{\mathcal{P}}}  \def\cR{{\mathcal{R}}} \def\cS{{\mathcal{S}}} \def\cT{{\mathcal{T}}} \def\cU{{\mathcal{U}}}     

\newcommand\Aut{\operatorname{Aut}}

\newcommand\Inn{\operatorname{Inn}}

\newcommand\Ker{\operatorname{Ker}}

\newcommand\Out{\operatorname{Out}}

\newcommand\Sym{\operatorname{Sym}}

\def\cc{{\curvearrowright}}

\newcommand{\so}{\mathfrak{s}}

\begin{document}
\title{Simple and large equivalence relations}
\author{Lewis Bowen\footnote{supported in part by NSF grant DMS-1500389, NSF CAREER Award DMS-0954606}}
%\author{University of Hawaii}
%\author[Lewis Bowen]{Lewis Bowen$\dagger$}
%\address{Department of Mathematics\\
%University of Hawai'i--Manoa\\
%} %one \address command per author
%\email{lpbowen@math.hawaii.edu}
%\thanks{$\dagger$ Supported in part by NSF grants DMS-??.}
\maketitle

\begin{abstract}
We construct  ergodic discrete probability measure preserving equivalence relations $\cR$ that has no proper ergodic normal subequivalence relations and no proper ergodic finite-index subequivalence relations. We show that every treeable equivalence relation satisfying a mild ergodicity condition and cost $>1$ surjects onto every countable group with ergodic kernel. Lastly, we provide a simple characterization of normality for subequivalence relations and an algebraic description of the quotient.
\end{abstract}


\noindent
\tableofcontents

\section{Introduction}
Let $(X,\mu)$ denote a standard Borel probability space and $\cR \subset X \times X$ a Borel equivalence relation. We say that $\cR$ is {\bf discrete} if for all $x\in X$, the $\cR$-class of $x$, denoted $[x]_\cR$, is countable. All equivalence relations considered in this note are discrete regardless of whether this is mentioned explicitly. We endow $\cR$ with two measures $\mu_L$ and $\mu_R$ satisfying: 
$$\mu_L(S) = \int |S_x|~d\mu(x), \quad \mu_R(S) = \int |S^y|~d\mu(y)$$
where 
$$S_x = \{y \in X:~(x,y) \in S\}, \quad S^y=\{x\in X:~(x,y) \in S\}.$$
We say that $\mu$ is {\bf $\cR$-quasi-invariant} if $\mu_L$ and $\mu_R$ are in the same measure class. We say $\mu$ is {\bf $\cR$-invariant} or $\cR$ is {\bf measure-preserving} if $\mu_L=\mu_R$. A subset $A \subset X$ is {\bf $\cR$-invariant} or {\bf $\cR$-saturated} if it is a union of $\cR$-classes. We say $\cR$ is {\bf ergodic} if for every measurable $\cR$-invariant subset $A \subset X$, $\mu(A) \in \{0,1\}$. In the sequel we will always assume $\mu$ is $\cR$-quasi-invariant.

%Let $[\cR]$ denote the {\bf full group} of $\cR$. This is the group of all Borel isomorphisms $\phi:X \to X$ with graph contained in $\cR$ in which we identify two isomorphisms if they agree almost everywhere. We say $\cR$ is probability-measure-preserving ({\bf probability-measure-preserving}) if $\phi_*\mu=\mu$ for every $\phi \in [\cR]$. 

A Borel subset $\cS \subset \cR$ is a {\bf subequivalence relation} if it is a Borel equivalence relation in its own right. If $\cS,\cS' \subset \cR$ are two subequivalence relations whose symmetric difference is null with respect to $\mu_L$ (or equivalently $\mu_R$) then we say $\cS$ and $\cS'$ agree almost everywhere (a.e.). From here on we will not distinguish between relations that agree almost everywhere. We say $\cS$ is {\bf proper} if it does not equal $\cR$ a.e. We usually write $\cS \le \cR$ to mean $\cS \subset \cR$ (a.e.) when $\cS$ is a subequivalence relation.

The concept of a normal subequivalence relation was introduced in \cite{FSZ88, FSZ89} where it was shown that if $\cS \subset \cR$ is normal then there is a natural quotient object, denoted $\cR/\cS$, which is a discrete measured groupoid. Moreover, if $\cS$ is ergodic then $\cR/\cS$ is a countable group and in this case we say $\cR$ {\bf surjects} onto $\cR/\cS$ and $\cR/\cS$ is a quotient of $\cR$. 

Unfortunately, the definition of normality given in \cite{FSZ88, FSZ89} is rather complicated. In \S \ref{sec:normal} we provide a simple characterization:  $\cS$ is {\bf normal} if and only if it is the kernel of a Borel morphism $c:\cR \to \cG$ where $\cG$ is a discrete Borel groupoid. We also show in \S \ref{sec:outer} that when $\cS$ is normal and ergodic in $\cR$ then $\cR/\cS$ is isomorphic with the full group $[\cR]$ modulo the normalizer of $[\cS]$ in $[\cR]$. This fact was probably known to the authors of \cite{FSZ88, FSZ89} but it is not explicit stated.

%This however, does not imply that $\cG$ is isomorphic to the quotient $\cR/\cS$, even assuming $c$ is class surjective.

If $\cS \subset \cR$ is an arbitrary subequivalence relation and $\cR$ is ergodic then, as shown in \cite{FSZ89}, there exists a number $N \in \N \cup \{\infty\}$ such that for a.e. $x\in X$, $[x]_\cR$ contains exactly $N$ $\cS$-classes. This number $N$ is called the {\bf index} of $\cS$ in $\cR$ and is denoted $N=[\cR:\cS]$. Our first main result:
\begin{thm}\label{thm:1}
There exists an ergodic discrete probability-measure-preserving equivalence relation $\cR$ such that $\cR$ does not contain any proper ergodic normal subequivalence relations. Moreover, $\cR$ does not contain any proper finite-index ergodic subequivalence relations.
\end{thm}

The proof of Theorem \ref{thm:1} is based on Popa's Cocycle Superrigidity Theorem \cite{Po07}. The ergodicity condition is necessary because if $\cP$ is any finite measurable partition of $X$ then the subequivalence relation $\cS$ defined by: $(x,y)\in \cS$ if and only $(x,y)\in \cR$ and $x,y$ are in the same part of $\cP$, has finite-index and is normal in $\cR$. Of course, $\cS$ is not ergodic if $\cP$ is nontrivial.

%A {\bf finite extension} of $\cR$ is an equivalence relation $\cU$ on $X$ such that $\cR \le \cU$ and $[\cU:\cR]$ is finite. The finite extension $\cU$ is trivial if $\cU=\cR$ a.e. 

\begin{remark}
Stefaan Vaes constructed the first explicit examples of type $II_1$ von Neumann algebras having only trivial finite index subfactors in \cite{Va08} by a twisted group-measure space construction.  Moreover it follows from \cite[Theorem 6.4]{Va08} that the orbit-equivalence relation of the generalized Bernoulli shift action $SL(2,\Q) \ltimes \Q^2 \cc (X_0,\mu_0)^{\Q^2}$ has no finite index ergodic subequivalence relations and no nontrivial finite extensions. Here $(X_0,\mu_0)$ is any nontrivial atomic probability space with unequal weights (so that it has trivial automorphism group). More generally, the proof of \cite[Theorem 6.4]{Va08} shows how to describe all finite index subequivalence relations, extensions and bimodules whenever cocycle super-rigidity applies. 
\end{remark}

We say that a measured equivalence relation $\cR$ is {\bf large} if for every countable group $G$ there exists an ergodic normal subequivalence relation $\cN \le \cR$ such that $\cR/\cN \cong G$. 

Next we prove that some treeable equivalence relations are large:
\begin{thm}\label{thm:3}
Suppose $\cR$ is a treeable ergodic equivalence relation on $(X,\mu)$ of cost $>1$ and there exists an ergodic primitive proper subequivalence relation $\cS\le \cR$. Then $\cR$ is large. 
\end{thm}
The terms treeable and primitive are explained in \S \ref{sec:treeable} below. (Primitive means the same as free factor; this notion was studied by Damien Gaboriau \cite{Ga00,Ga05}). For example, the orbit-equivalence relation of any Bernoulli shift over a rank $\ge 2$ free group satisfies the hypothesis above and hence is large. It is an open question whether every ergodic treeable equivalence relation with cost $>1$ satisfies the hypotheses of Theorem \ref{thm:3}.  In particular it is unknown whether every such equivalence relation surjects every countable group. In unpublished work, Clinton Conley, Damien Gaboriau, Andrew Marks and Robin Tucker-Drob have proven that any treeable {\em strongly ergodic} pmp equivalence relation satisfies the hypothesis of Theorem \ref{thm:3} and therefore surjects onto every countable group.

%Recall that a group $G$ is {\bf large} if it surjects onto a nonabelian free group. So we say that $G$ is {\bf measurably large} if there exists an ergodic essentially free probabilty-measure-preserving action $G \cc (X,\mu)$ such that if $\cR$ is the orbit-equivalence relation then there exists a normal ergodic subequivalence relation $\cN \le \cR$ such that $\cR/\cN$ is a nonabelian free group.

%\begin{thm}
%Suppose $G$ is measurable large. Then  there exists an ergodic essentially free probabilty-measure-preserving action $G \cc (Y,\nu)$ such that $\cR$ is large. % the orbit-equivalence relation then there exists a normal ergodic subequivalence relation $\cN \le \cR$ such that $\cR/\cN \cong \F_\infty$, the infinite rank free group. Thus for every countable group $H$ there is some ergodic normal subequivalence relation $\cK\le \cR$ such that $\cR/\cK \cong H$. 
%\end{thm}

%\begin{example}
%Suppose $G$ is a lattice in $SO(n,1)$ or $SU(n,1)$. Then $G$ is measure
%\end{example}

{\bf Organization:} In \S \ref{sec:normal} we provide a simple characterization for normality of a subequivalence relation. \S \ref{sec:quotient} provides an algebraic description for the quotient $\cR/\cN$. \S \ref{sec:Bernoulli} reviews generalized Bernoulli shifts and Popa's Cocycle Superrigidity Theorem \cite{Po07}. The latter is used in \S \ref{sec:simple} to prove Theorem \ref{thm:1}.  The last section \S \ref{sec:treeable} proves Theorem \ref{thm:3}. The proof is mostly independent of the rest of the paper.

{\bf Acknowledgements}. Thanks to Vadim Alekseev for suggesting the use of Popa's Cocycle Superrigidity Theorem to prove Theorem \ref{thm:1}; to Alekos Kechris for some inspiring questions, many helpful conversations and many corrections; to Robin Tucker-Drob for helpful conversations; to Pierre-Emmanuel Caprace for conversation about Kac-Moody groups that, with super-rigidity results, yield more examples of equivalence relations without finite extensions.

% An earlier version of this paper contained an appendix by Pierre-Emmanuel Caprace that constructed a Kac-Moody group with special properties. In conjunction with a rigidity results of Popa and Vaes, this was used to construct an orbit equivalence relation which has no finite-extensions. Upon learning of Stefaan Vaes' stronger examples in \cite{Va08}, I decided to remove 

\section{A simple criterion for normality}\label{sec:normal}

%We assume $\mu$ is $\cR$-invariant. 
\begin{defn}[Choice functions]
Let $\cR$ be an ergodic discrete Borel equivalence relation on $(X,\mu)$. Let $\cS \subset \cR$ be a Borel subequivalence relation and let $N=[\cR:\cS]$. A {\bf family of choice functions} is a set $\{\phi_j\}_{j=1}^N$ of Borel functions $\phi_j:X \to X$ such that for each $x\in X$, $\{ [\phi_j(x)]_\cS\}_{j=1}^N$ is a partition of $[x]_\cR$. \cite[Lemma 1.1]{FSZ89} shows that a family of choice functions exists.
\end{defn}

\begin{thm}\label{thm:normal}
Let $\cR$ be an ergodic discrete Borel equivalence relation on $(X,\mu)$. Let $\cS \subset \cR$ be a subequivalence relation.  The following are equivalent.
\begin{enumerate}
\item $\cS$ is normal in $\cR$ in the sense of \cite[Definition 2.1]{FSZ89}.
\item There are choice functions $\{\phi_j\}$ for $\cS \subset \cR$ with $\phi_j \in \rm{End}_\cR(\cS)$ for all $j$. This means that if $(x,y) \in \cS$ then $(\phi_j(x),\phi_j(y)) \in \cS$.
\item The extension $\rho:\hat{\cR} = \cS \times_\sigma J \to \cS$ is normal in the sense of Zimmer \cite{Zi76};
\item There is a discrete, ergodic measured groupoid $(\cH,\nu)$ and a homomorphism $\theta:\cR \to \cH$ such that
\begin{enumerate}
\item $\ker(\theta)=\cS$;
\item $\theta$ is class-surjective in the following sense: for any $h \in \cH$ and $x\in X$ with $\theta(x)$ equal to the source of $h$,  there exists $y\in [x]_\cR$ with $\theta(y,x)=h$;
\item for any discrete ergodic measured groupoid $(\cH',\nu')$ and homomorphism $\theta':\cR \to \cH'$ with $\ker(\theta') \supset \cS$ there is a homomorphism $\kappa:\cH \to \cH'$ with $\kappa \theta = \theta'$;
\end{enumerate}
\item there is a discrete Borel groupoid $\cG$ and a Borel homomorphism $c:\cR \to \cG$ with $\cS=\Ker(c)$.
\end{enumerate}
\end{thm}

\begin{proof}
The equivalence of the first four statements is \cite[Theorem 2.2]{FSZ89}. Clearly (4) implies (5). So we need only show that (5) implies (2). So let $\cG$ be a discrete Borel groupoid with unit space $\cG^0$, source and range maps $\so,\mathfrak{r}:\cG \to \cG^0$. As in the proof of the Feldman-Moore Theorem \cite[Theorem 1]{FM77}, there exists a countable family of Borel functions $\{f_j\}_{j=1}^\infty$, $f_j:\cG^0 \to \cG$ such that for every $x\in \cG^0$,
$$\{f_j(x)\}_{j \in \N} = \so^{-1}(x).$$
Let $c:\cR\to \cG$ be a Borel homomorphism. Because $\cR$ is ergodic there exists $n \in \N \cup \{\infty\}$ be such that  $\{c(x,y):~y \in [x]_\cR\}$ has cardinality $n$ for a.e. $x$. Let $\psi_j:X \to X$ be a Borel map such that if $\psi_j(x)=y$ then $y \in [x]_\cR$ and $c(x,y)=f_j(c(x,y))$. 

For each $x\in X$, define $\phi_1(x) = \psi_{1}(x)$. For $i>1$ inductively define $\phi_i(x)=\psi_{j}(x)$ where $j$ is the smallest number such that there does not exist $k<i$ with $\phi_k(x)=\psi_{j}(x)$. Then $\{\phi_j\}_{j=1}^n$ is family of choice functions satisfying (2). 

\end{proof}

\section{The quotient group}\label{sec:quotient} 

%\section{Outer automorphisms and relations}\label{sec:outer}

The purpose of this section is to provide an algebraic description of the quotient $\cR/\cN$ when $\cN\vartriangleleft \cR$ is normal. 

To be precise, let $\cR$ denote an ergodic probability-measure-preserving discrete equivalence relation on a probability space $(X,\mu)$. Let $\Aut(X,\mu)$ be the group of all measure-preserving automorphisms $\phi:X \to X$. We implicitly identify two automorphisms that agree almost everywhere. Let $\Aut(\cR)$  be the subgroup of all $\phi \in \Aut(X,\mu)$ such that $x \cR y \Rightarrow \phi(x) \cR \phi(y)$ (for $\mu_L$-a.e. $(x,y)$). Also let $[\cR]=\Inn(\cR)$ be the subgroup of all $\phi \in \Aut(\cR)$ such that $x \cR \phi(x)$ for a.e. $x$. Then $[\cR]$ is normal in $\Aut(\cR)$, so we may consider the quotient $\Out(\cR):=\Aut(\cR)/[\cR]$. 

In the sequel we use the word `countable' to mean `countable or finite'. 

Let $\Gamma$ be a countable subgroup of $\Aut(\cR)$. Let $\cR_\Gamma:=\langle\cR,\Gamma\rangle$ denote the smallest equivalence relation on $X$ such that $\cR \subset \cR_\Gamma$ and $(x,\gamma x) \in \cR_\Gamma$ for all $x\in X$ and $\gamma \in \Gamma$. Observe that this is a discrete probability-measure-preserving ergodic equivalence relation and $\cR$ is normal in $\cR_\Gamma$.

\begin{lem}
If $\Gamma,\Lambda \le \Aut(\cR)$ are countable subgroups and $\Gamma [\cR] = \Lambda [\cR]$ then $\cR_\Gamma=\cR_\Lambda$. 
\end{lem}

\begin{proof}
This is straightforward.
\end{proof}

So if $\Gamma \le \Out(\cR)$ is any countable subgroup then we may define $\cR_\Gamma := \cR_{\Gamma'}$ where $\Gamma' \le \Aut(\cR)$ is any countable subgroup such that $\Gamma = \Gamma'[\cR]$. The next lemma follows immediately.

\begin{lem}\label{lem:order}
If $\Gamma \le \Lambda \le \Out(\cR)$ are countable subgroups then $\cR_\Gamma \le \cR_\Lambda$.
\end{lem}

 \begin{thm}\label{thm:correspond}
Let $\cR \le \cU$ be ergodic discrete probability-measure-preserving equivalence relations on $(X,\mu)$. Suppose $\cR$ is normal in $\cU$. Then there exists a countable subgroup $\Gamma \le \Out(\cR)$ such that $\cU = \cR_\Gamma$. Moreover, $\cR_\Gamma/\cR$ is isomorphic to $\Gamma$. In particular, $[\cR_\Gamma:\cR]= |\Gamma|$. 
\end{thm}

\begin{proof}
This follows from  \cite[Theorems 2.12 and 2.13]{FSZ89}. 

%that there exists a countable subgroup $\Gamma \le \Out(\cR)$ such that $\cU=\cR_\Gamma$. Clearly, $[\cR_\Gamma:\cR] = |\cR_\Gamma/\cR|$. So it suffices to prove $\cR_\Gamma/\cR$ is isomorphic to $\Gamma$.

%This follows from \cite[Theorems 2.12 and 2.13]{FSZ89}.
\end{proof}

\begin{cor}\label{cor:normalizer}
Let $\cR \le \cU$ be ergodic discrete probability-measure-preserving equivalence relations on $(X,\mu)$. Suppose $\cR$ is normal in $\cU$. Then
$$\cU/\cR \cong N_\cU(\cR)/[\cR]$$
where
$$N_\cU(\cR) := \{g \in [\cU]:~ g[\cR]g^{-1}=[\cR]\}.$$
Moreover  $N_\cU(\cR)= [\cU] \cap \Aut(\cR).$ % \{g\in [\cU]:~(x,y) \in \cR \Rightarrow (gx,gy) \in \cR\}.$$
\end{cor}

\begin{proof}
We prove the last claim first. So suppose $g \in N_\cU(\cR)$ and $x\in X$. Then for any $\phi \in [\cR]$ we must have $g\phi g^{-1} \in [\cR]$. This implies $(x , g\phi g^{-1} x) \in \cR$ which implies, by replacing $x$ with $gx$, that $(gx, g\phi x) \in \cR$. Since $\phi$ is arbitrary and $[\cR]$ acts transitively on each $\cR$-class, this implies $g \in \Aut(\cR)$. Thus $N_\cU(\cR) \subset [\cU] \cap \Aut(\cR).$

Now suppose $g\in [\cU] \cap \Aut(\cR)$. If $\phi \in [\cR]$ then $(x,\phi x) \in \cR$. This implies $(gx, g\phi x) \in \cR$. By replacing $x$ with $g^{-1}x$ we obtain $(x,g\phi g^{-1}x) \in \cR$  which implies $g\phi g^{-1} \in [\cR]$ (since $x$ is arbitrary). Thus $g \in N_\cU(\cR)$. This proves $N_\cU(\cR)= [\cU] \cap \Aut(\cR).$

By Theorem \ref{thm:correspond} there exists a countable subgroup $\Gamma \le \Out(\cR)$ such that $\cR_\Gamma=\cU$ and $\cU/\cR\cong \Gamma$. Let $\tilde{\Gamma} \le \Aut(\cR)$ be the inverse image of $\Gamma$ under the quotient map $\Aut(\cR) \to \Aut(\cR)/[\cR] = \Out(\cR)$. Since $\cU = \cR_\Gamma$ we must have $\tilde{\Gamma} \le [\cU]$ and therefore $\tilde{\Gamma} \le N_\cU(\cR)$ which implies $\Gamma \le N_\cU(\cR)/[\cR]$. 

On the other hand, we clearly have $\cR_{N_{\cU}(\cR)} \le \cU=\cR_\Gamma$. So Lemma \ref{lem:order} implies $N_\cU(\cR)/[\cR] \le \Gamma$. Theorem \ref{thm:correspond} now implies $N_\cU(\cR)/[\cR] = \Gamma \cong \cU/\cR.$

%$g \in N_\cU(\cR) \le \Aut(\cR)$

% and $\Lambda \le \Out(\cR)$ the group generated by $g [\cR]$. Then $

%To obtain a contradi

%To finish the corollary, it suffices to prove that the image of $N_\cU(\cR)$ in $\Out(\cR)$ 

%Let $\Ad:N_\cU(\cR) \to \Aut(\cR)$ be the map
%$$\Ad(\phi) x = $$

% It suffices to prove that $[\cU] \cap $

\end{proof}

\section{Generalized Bernoulli shifts and cocycle superrigidity}\label{sec:Bernoulli}

Let $G$ be a countable group, $I$ a countable set on which $G$ acts and $(X_0,\mu_0)$ a standard probability space. We let $X_0^I$ be the set of all functions $x:I \to X_0$ and $\mu_0^I$ the product measure on $X_0^I$. Then $G \cc X_0^I$ by $(gx)(i) =x(g^{-1}i)$. This action preserves the measure $\mu_0^I$. The action $G \cc (X_0,\mu_0)^I$ is a {\bf generalized Bernoulli shift}. 

Our interest in these actions stems from Popa's Cocycle Super-rigidity Theorem. To explain, let $G \cc (X,\mu)$ be a probability-measure-preserving action. A {\bf cocycle} into a countable group $H$ is a Borel map $c: G \times X \to H$ such that
$$c(g_1g_2,x) = c(g_1,g_2x)c(g_2,x).$$ 
Alternatively, if $G \cc (X,\mu)$ is essentially free then we can identify $G \times X$ with the orbit-equivalence relation, denoted $\cR$, via $(g,x) \mapsto (gx,x)$. In this way, we can think of the cocycle as map from $\cR$ to  $H$. We say the action is {\bf cocycle superrigid} if for every such cocycle there is a homomorphism $\rho:G \to H$ and a Borel map $\phi:X \to H$ such that
$$c(g,x)= \phi(gx)^{-1} \rho(g) \phi(x).$$

The next result is a special case of celebrated theorem due to S. Popa \cite{Po07} (see also \cite[Theorem 3.2 and Proposition 2.3]{PV08}).

\begin{thm}\label{thm:popa}
Suppose every orbit of $G \cc I$ is infinite.  If $G$ has property (T) then the generalized Bernoulli shift $G \cc (X_0,\mu_0)^I$ is cocycle superrigid.
\end{thm}

\section{A simple equivalence relation}\label{sec:simple}

\begin{thm}\label{thm:superrigid}
Suppose $G \cc (X,\mu)$ is an essentially free measure-preserving ergodic action of a countably infinite group $G$ on a standard probability space $(X,\mu)$. Let $\cR = \{(x,gx):~x\in X, g\in G\}$ be the orbit equivalence relation. Suppose $G \cc (X,\mu)$ is cocycle superrigid. If $G$ is simple then $\cR$ has no proper ergodic normal subequivalence relations. If $G$ has no nontrivial finite quotients then $\cR$ has no proper ergodic normal finite-index subequivalence relations.

\end{thm}

\begin{proof}
 Let $\cR$ be the orbit-equivalence relation of the action $G \cc (X,\mu)$.  Let $\cN\vartriangleleft \cR$ be an ergodic normal subequivalence relation and $c:\cR \to \cR/\cN$ the canonical cocycle. Since the action is cocycle superrigid, there exists  a Borel map $\phi:X \to H:=\cR/\cN$ and a homomorphism $\rho:G \to H$ such that
$$c(gx,x) = \phi(gx)^{-1} \rho(g) \phi(x).$$
\noindent {\bf Claim 1}. If $\rho$ is trivial then $\cN=\cR$.

\begin{proof}[Proof of Claim 1]
For every $g\in \cR/\cN$, $\phi^{-1}(g) \subset X$ is $\cN$-invariant. Because $\cN$ is ergodic, this implies $\phi$ is essentially constant which implies $\cN=\cR$.
\end{proof}

\noindent {\bf Claim 2}. $\rho$ is non-injective.

\begin{proof}[Proof of Claim 2]
To obtain a contradiction, suppose $\rho$ is injective. We claim that $\cN$ is finite. To see this, let 
$$X_g=\{x \in X:~ (gx,x) \in \cN\} = \{x\in X:~ \phi(gx)\phi(x)^{-1}=\rho(g)\}.$$
Also, for $g \in G, h \in \cR/\cN$ let
$$X_{g,h} = \{x\in X_g:~ \phi(x)=h\} = \{x\in X:~\phi(x)=h, \phi(gx) = \rho(g)h\}.$$
Because $\rho$ is injective, for any fixed $h$, the sets $\{gX_{g,h}:g\in G\}$ are pairwise disjoint. Therefore $\sum_{g\in G} \mu(X_{g,h}) \le 1$. By the Borel-Cantelli Lemma, almost every $x$ is contained in at most finitely many of the $X_{g,h}$'s (for fixed $h$). However for each $x\in X$ there is at exactly one $h$ such that $x$ is contained in some $X_{g,h}$. So, in fact, $x$ is contained in at most finitely many $X_{g,h}$'s (allowing $g$ and $h$ to vary). Since $X_g=\cup_h X_{g,h}$ this implies that a.e. $x$ is contained in at most finitely many $X_g$'s which implies that for a.e. $x$, the $\cN$-equivalence class $[x]_\cN$ is finite.

Because $G$ is infinite and $G\cc (X,\mu)$ is essentially free and ergodic, $\mu$ is nonatomic. Because $\cN$ is finite and $\mu$ is nonatomic, $\cN$ is not ergodic. This contradiction proves that $\rho$ is non-injective.
\end{proof}

If $G$ is simple, either $\rho$ is trivial or injective and so the claims above finish the proof. If $G$ does not have any nontrivial finite quotients and $\cR/\cN$ is finite then $\rho:G \to \cR/\cN$ must be trivial. So Claim 1 finishes the proof.

\end{proof}

\begin{defn}[Compressions]
Let $\cR \subset X \times X$ be a probability-measure-preserving Borel equivalence relation on a probability space $(X,\mu)$. If $Y \subset X$ has positive measure then we let $\cR_Y:=\cR \cap (Y \times Y)$ denote the {\bf compression of $\cR$ by $Y$}. It is an equivalence relation on $Y$.
\end{defn}

\begin{lem}\label{lem:compression}
Let $Y \subset X$ be a Borel set with positive measure. Let $\cS\le \cR_Y$ be a subequivalence relation. Then there exists a subequivalence relation $\cT\le \cR$ such that $\cT_Y=\cS$. Moreover, if $\cS$ is ergodic then $\cT$ is ergodic and if $\cS$ is normal in $\cR_Y$ then $\cT$ is normal in $\cR$.
\end{lem}

\begin{remark}
$\cT$ is not unique. Moreover, even if $\cS$ is ergodic and normal there may exist subequivalence relations $\cT'$ such that $\cT'_Y=\cS$ but $\cT'_Y$ is neither ergodic nor normal.
\end{remark}

\begin{proof}
Because $\cR$ is ergodic there exists a Borel map $\phi:X \to Y$ with graph contained in $\cR$ such that $\phi(y)=y$ for every $y\in Y$. Define the subequivalence relation $\cT$ by $x\cT y$ iff $\phi(x)\cS\phi(y)$. In other words, if $\Phi:\cR \to \cR_Y$ is the map $\Phi(x,y) = (\phi(x),\phi(y))$ then $\cT  = \Phi^{-1}(\cS)$. This implies that $\cT $ is Borel. It is easy to check that $\cT $ is a subequivalence relation and $\cT _Y=\cS$.

Suppose that $\cS$ is ergodic. Let $A \subset X$ be an $\cT $-saturated set of positive measure. Observe that $A = \phi^{-1}(\phi(A))$ by definition of $\cT $. Also $\phi(A)$ is $\cS$-saturated. Therefore $\phi(A)=Y$ since $\cS$ is ergodic. So $A=\phi^{-1}\phi(A) =X$. Because $A$ is arbitrary, $\cT $ is ergodic.

Suppose that $\cS$ is normal. Then there exists a groupoid morphism $c:\cR_Y \to \cG$ such that $\cS=\Ker(c)$. Define $c':\cR \to \cG$ by $c'(x,y) = c(\phi(x),\phi(y))$. Observe that  
$$c'(x,y)c'(y,z) = c(\phi(x),\phi(y))c(\phi(y),\phi(z)) = c(\phi(x),\phi(z)) = c'(x,z).$$
 So $c'$ is a cocycle. If $(x,y) \in \Ker(c')$ then $c(\phi(x),\phi(y)) \in \cG^0$ which implies $(\phi(x),\phi(y)) \in \Ker(c) = \cS$ which implies $(x,y) \in \cT $. So $\Ker(c') \subset \cT$. On the other hand, if $(x,y) \in \cT$ then $(\phi(x),\phi(y)) \in \cS=\Ker(c)$ which implies $(x,y) \in \Ker(c')$. So $\cT=\Ker(c')$ is normal by Theorem \ref{thm:normal}.
  
\end{proof}

\begin{prop}\label{prop:finite-index}
Let $\cR$ be an ergodic probability-measure-preserving equivalence relation with a finite-index ergodic subequivalence relation $\cS\le \cR$. Then $\cR$ has a finite-index ergodic normal subequivalence relation $\cN$ with $\cN\le \cS$.
\end{prop}

\begin{proof}
Let $n=[\cR:\cS]$ denote the index of $\cS$ in $\cR$. Let $\phi:\cR \to \{1,\ldots, n\}$ be any Borel function satisfying
\begin{itemize}
\item for a.e. $x\in X$, $\phi(x,x)=1$
\item for a.e. $(x,y), (x,z) \in \cR$ with $(y,z)\in \cS$, $\phi(x,y)=\phi(x,z)$
\item for a.e. $x\in X$, the map $y \mapsto \phi(x,y)$ surjects onto $\{1,\ldots, n\}$. So this map is a bijection from the set of $\cS$-classes in $[x]_\cR$ to $\{1,\ldots,n\}$. 
\end{itemize}
Define a cocycle $\alpha:\cR \to \Sym(n)$ (the symmetric group of $\{1,\ldots, n\}$) by 
$$\alpha(x,y)(k) = \phi(x,z)$$
where $z\in [x]_\cR$ is any element satisfying $\phi(y,z)=k$. Let $\cK$ be the kernel of this cocycle. This is a finite-index normal subequivalence relation and $\cK \le \cS$ but $\cK$ might not be ergodic. However, it can have at most finitely many ergodic components (this is true for any finite-index subequivalence relation). Let $Y \subset X$ be an ergodic component of $\cK$. So $\cK _Y$ is an ergodic finite-index normal subequivalence relation of $\cR _Y$. By Lemma \ref{lem:compression} there exists an ergodic normal finite index subequivalence relation $\cN\le \cR$ such that $\cN _Y = \cK _Y$. Since $\cN$ and $\cS$ are ergodic and $\cN _Y \le \cS _Y$ we must have that $\cN \le \cS$. 

%check: $\alpha(x,y)(\alpha(y,w)(k)) = \alpha(x,y)c(w,z) = c(x,t) where z is anything with \alpha(y,z)=k and t is anything with alpha(
\end{proof}

\begin{proof}[Proof of Theorem \ref{thm:1}]
Let $G$ be a simple property (T) group.  Quoting from \cite{Th10}: there are two sources of simple groups with Kazhdan's property (T). Such groups appear for example as lattices in certain Kac-Moody groups, see \cite{CR06}. Much earlier, it was also shown by Gromov (\cite{Gr87}) that every hyperbolic group surjects onto a Tarski monster, i.e. every proper subgroup of this quotient is finite cyclic; in particular: this quotient group is simple and is a Kazhdan group if the hyperbolic group was a Kazhdan group.

Let $(X_0,\mu_0)$ be a nontrivial Borel probability space and $G \cc (X,\mu):=(X_0,\mu_0)^G$ the Bernoulli shift action. By Popa's Cocycle Superrigidity Theorem \ref{thm:popa}, $G \cc (X,\mu)$ is cocycle superrigid. So Theorem \ref{thm:superrigid} implies $\cR$ has  no ergodic proper normal subequivalence relations. Proposition \ref{prop:finite-index} implies  $\cR$ has  no ergodic proper finite-index subequivalence relations

\end{proof}

\section{Treeable equivalence relations}\label{sec:treeable}

\begin{defn}
A {\bf graphing} of an equivalence relation $\cR \subset X \times X$ is a Borel subset $G \subset X \times X$ such that $\cR$ is the smallest equivalence relation containing $G$ and $G$ is symmetric: $(x,y) \in G \Rightarrow (y,x) \in G$. The {\bf local graph} of $G$  at $x$ is denoted by $G_x$. It has vertex set $[x]_\cR$ and edges $\{y,z\}$ where $y,z \in [x]_\cR$ and $(y,z) \in G$. So $G$ is a graphing if and only if it is symmetric and all local graphs are connected. A graphing is a {\bf treeing} if all of its local graphs are trees.
\end{defn}

\begin{defn}
Let $\cR$ be an ergodic treeable equivalence relation. A subequivalence relation $\cS\le \cR$ is {\bf primitive} if there exist treeings $G_\cS, G_\cR$ of $\cS$ and $\cR$ such that $G_\cS \subset G_\cR$. This means the same as free factor as used in \cite{Ga00,Ga05}. %This definition is motivated by the group theory definition of primitive: if $F$ is a free group then a subgroup $F_0<F$ is primitive if there is a subgroup $F_1<F$ such that $F=F_0*F_1$. 
\end{defn}

\begin{example}
If $F=\langle S \rangle$ is a free group with free generating set $S \subset F$ and $F \cc (X,\mu)$ is an essentially free action then $G_F=\{(x,sx), (sx,x):~x \in X, s\in S\}$ is a treeing of the orbit-equivalence relation $\cR$. Moreover if $g \in S$ and $\cS$ is the orbit-equivalence relation generated by $\{g^n\}_{n\in \Z}$ then $\cS$ is primitive in $\cR$ since $G_\cS=\{ (x,gx), (gx,x):~ x\in X\}$ is a treeing of $\cS$ and $G_\cS \subset G_F$. More generally, if $g$ is primitive in $F$ (this means that it is contained in some free generating set of $F$) and $\cS$ is the orbit-equivalence relation of $\{g^n\}_{n\in \Z}$, then $\cS$ is primitive in $\cR$.
\end{example}

Before proving Theorem \ref{thm:3} we need a lemma.

\begin{lem}\label{lem:basic2}
Suppose $\Gamma$ is a countable group and $c:\cR \to \Gamma$ is a cocycle such that $\Ker(c)\le \cR$ is ergodic. Let
$$\Lambda=\{g\in \Gamma:~ \mu_L( \{ (x,y) \in \cR:~c(x,y)=g\})>0 \}.$$
Then $\Lambda$ is a subgroup of $\Gamma$ and $\cR/\Ker(c)$ is isomorphic to $\Lambda$.
\end{lem}

\begin{proof}
For $x\in X$, let $\Gamma_x = \{c(x,y):~y \in [x]_\cR\}$. If $(x,z) \in \Ker(c)$ then $c(x,y)=c(z,y)$. So $\Gamma_x=\Gamma_z$. Since $\Ker(c)$ is ergodic, this implies the existence of a subset $\Gamma' \subset \Gamma$ such that $\Gamma'=\Gamma_x$ for a.e. $x$. Observe that since $c(x,y) \in \Gamma_x$, $c(y,x)=c(x,y)^{-1} \in \Gamma_y$. Thus $\Gamma'$ is invariant under inverse. Also if $c(x,y) \in \Gamma_x$ and $c(y,z) \in \Gamma_y$ then $c(x,z)=c(x,y)c(y,z) \in \Gamma_x$. So $\Gamma'$ is a subgroup. By ergodicity again, $\Lambda=\Gamma'$.  So without loss of generality, we may assume $\Lambda=\Gamma$.

It follows from \cite[Theorem 2.2]{FSZ89} that there is a homomorphism $\theta': \cR/\Ker(c) \to \Gamma$ such that if $\theta: \cR \to \cR/\Ker(c)$ is the canonical morphism then 
$$\theta' \theta = c.$$
Since $c$ and $\theta$ have the same kernel, $\theta'$ must be injective. Since $\Lambda=\Gamma$, it is also surjective and so $\cR/\Ker(c)$ is isomorphic to $\Lambda$.

\end{proof}

\begin{thm}\label{thm:key}
Suppose $\cR$ is a treeable ergodic probability-measure-preserving equivalence relation on $(X,\mu)$ of cost $>1$ and there exists a subequivalence relation $\cS \le \cR$ that is primitive, ergodic and proper. Then  $\cR$ surjects onto every countable group.
\end{thm}

\begin{proof}
Because $\cS\le \cR$ is primitive, there exist treeings $G_\cS \subset G_\cR$ of $\cS$ and $\cR$. Because $\cS$ is proper, $\mu_L(\cR\setminus \cS)>0$ and therefore $\mu_L(G_\cR \setminus G_\cS)>0$. Let $c:G_\cR \setminus G_\cS \to \F_\infty$ be any measurably surjective map such that $c(x,y) = c(y,x)^{-1}$ wherever this is defined. Here $\F_\infty$ denotes the free group of countable rank. We extend $c$ to $G_\cS$ by $c(x,y)=e$ for any $(x,y)\in \cS$. Now $c$ is defined on all of $G_\cR$. Because $G_\cR$ is a treeing there is a unique extension of $c$ to a cocycle $c:\cR \to \F_\infty$.

By definition $\Ker(c)$ contains $\cS$. Because $\cS$ is ergodic, this implies $\Ker(c)$ is ergodic. Lemma \ref{lem:basic2} now implies $\cR/\Ker(c) \cong \F_\infty$. 

Now let $\Lambda$ be an arbitrary countable group and $\phi:\F_\infty \to\Lambda$ a surjective homomorphism. Let $c':\cR \to \Lambda$ be the cocycle $c'(x,y)=\phi(c(x,y))$. Since $\Ker(c)$ is ergodic, $\Ker(c')$ is ergodic. So Lemma \ref{lem:basic2} implies $\cR/\Ker(c') \cong \Lambda$. 
\end{proof}

\begin{example}
If $\cR$ is the orbit-equivalence relation of a Bernoulli shift action of  $\F_n$ ($n\ge 2$)  then every generator of $\F_n$ acts ergodically. Therefore, $\cR$ satisfies the hypotheses of Theorem \ref{thm:key}. 
\end{example}

\begin{remark}
 Clinton Conley, Damien Gaboriau, Andrew Marks and Robin Tucker-Drob have proven that any treeable {\em strongly ergodic} pmp equivalence relation satisfies the hypothesis of Theorem \ref{thm:key}. This work has not yet been published.
\end{remark}

\begin{conj}\label{conj:erg-tree}
Let $\cR$ be an ergodic treeable equivalence relation of cost $>1$. Then there exists an ergodic element $f \in [\cR]$ such that the subequivalence relation generated by $f$ is primitive in $\cR$. 
\end{conj}

It follows from Theorem \ref{thm:key} that if the above conjecture is true then every treeable ergodic probability-measure-preserving equivalence relation surjects every countable group.

\appendix 

%\section{Notes}

%* shorten biblio

%* fix last 2-3 paragraphs of intro

%* fix outer automorph section.


{\small 

\begin{thebibliography}{10000000}


%\bibitem[Bo10a]{Bo10a} L. Bowen. \textit{A measure-conjugacy invariant for actions of free groups}. Ann. of Math., vol. 171 (2010), No. 2, 1387--1400.

%\bibitem[CI10]{CI10} I. Chifan, A. Ioana, Ergodic subequivalence relations induced by a Bernoulli action, Geom. Funct. Anal. Vol. 20 (2010), 53--67.

%\bibitem[AB08]{AB08} Peter Abramenko and Kenneth S. Brown, Buildings, Graduate Texts in Mathematics, vol. 248, Springer, New York, 2008. Theory and applications. 

%\bibitem[Ba95]{Ba95} Werner Ballmann, Lectures on spaces of nonpositive curvature, DMV Seminar, vol. 25, Birkhäuser Verlag, Basel, 1995. With an appendix by Misha Brin. 

%\bibitem[CF10]{CF10} Pierre-Emmanuel Caprace and Koji Fujiwara, Rank-one isometries of buildings and quasi-morphisms of Kac-Moody groups, Geom. Funct. Anal. 19 (2010), no. 5, 1296--1319. 

%\bibitem[CH15]{CH15} Pierre-Emmanuel Caprace and David Hume, Orthogonal forms of Kac--Moody groups are acylindrically hyperbolic, 2015. Ann. Inst. Fourier, to appear. 
%\bibitem[CM15]{CM15} Pierre-Emmanuel Caprace and Nicolas Monod, An indiscrete Bieberbach theorem: from amenable CAT(0) groups to Tits buildings, 2015. preprint arXiv:1502.04583. 
%\bibitem[CM06]{CM06} Pierre-Emmanuel Caprace and Bernhard M\"uhlherr, Isomorphisms of Kac-Moody groups which preserve bounded subgroups, Adv. Math. 206 (2006), no. 1, 250--278. 
\bibitem[CR06]{CR06} Pierre-Emmanuel Caprace and Bertrand R\'emy, Simplicit\'e abstraite des groupes de Kac-Moody non affines, C. R. Math. Acad. Sci. Paris 342 (2006), no. 8, 539--544.
%\bibitem[CR09]{CR09} Pierre-Emmanuel Caprace and Bertrand R\'emy, Simplicity and superrigidity of twin building lattices, Invent. Math. 176 (2009), no. 1, 169--221.
%\bibitem[Da94]{Da94} Michael W. Davis, Buildings are CAT(0), Geometry and cohomology in group theory (Durham, 1994), London Math. Soc. Lecture Note Ser., vol. 252, Cambridge Univ. Press, Cambridge, 1998, pp. 108--123.
%\bibitem[DJ02]{DJ02}  Jan Dymara and Tadeusz Januszkiewicz, Cohomology of buildings and their automorphism groups, Invent. Math. 150 (2002), no. 3, 579--627.
%\bibitem[EJ10]{EJ10} Mikhail Ershov and Andrei Jaikin-Zapirain, Property (T) for noncommutative universal lattices, Invent. Math. 179 (2010), no. 2, 303--347. 
\bibitem[FM77]{FM77} Jacob Feldman, and Charles C. Moore,  Ergodic equivalence relations and von Neumann algebras I. Trans. Amer. Math. Soc., \textbf{234}, (1977),289-- 324. 
\bibitem[FSZ88]{FSZ88} Jacob Feldman, Colin Sutherland and Robert J. Zimmer, Normal subrelations of ergodic equivalence relations. Miniconferences on harmonic analysis and operator algebras (Canberra, 1987), 95--102, Proc. Centre Math. Anal. Austral. Nat. Univ., 16, Austral. Nat. Univ., Canberra, 1988.

\bibitem[FSZ89]{FSZ89} Jacob Feldman, Colin Sutherland and Robert J. Zimmer, Subrelations of ergodic equivalence relations, Ergodic Theory Dynam. Systems 9 (1989), 239--269.


\bibitem[Ga00]{Ga00} Damien Gaboriau. Co\^ut des relations d' \'equivalence et des groupes. Invent. Math., 139(1):41--98, 2000.

\bibitem[Ga05]{Ga05} Damien Gaboriau. Examples of groups that are measure equivalent to the free group. Ergodic Theory Dynam. Systems, 25(6):1809--1827, 2005.


\bibitem[Gr87]{Gr87} Misha Gromov, Hyperbolic groups, Essays in group theory, Math. Sci. Res. Inst. Publ., vol. 8, Springer, New York, 1987, pp. 75--263.


%\bibitem[Kl99]{Kl99} Bruce Kleiner, The local structure of length spaces with curvature bounded above, Math. Z. 231 (1999), no. 3, 409--456. 

\bibitem[KM04]{KM04} Alexander Kechris and Ben Miller, Topics in orbit equivalence, Lecture Notes in Mathematics, vol. 1852, Springer, 2004.


%\bibitem[Ma14]{Ma14} Timoth\'ee Marquis, Abstract simplicity of locally compact Kac-Moody groups, Compos. Math. 150 (2014), no.4, 713--728.

%\bibitem[Remy]{Remy} 

%\bibitem[Mar]{Mar} Timoth\'ee Marquis, Abstract simplicity of locally compact Kac-Moody groups, Compos. Math. 150 (2014), no. 4, 713--728. 
\bibitem[Po07]{Po07} Sorin Popa,  Cocycle and orbit equivalence superrigidity for malleable actions of w-rigid groups. Invent. Math. 170 (2007), no. 2, 243--295.

\bibitem[PV08]{PV08} Sorin Popa and Stefaan Vaes, Strong rigidity of generalized Bernoulli actions and computations of their symmetry groups. Adv. Math. 217 (2008), no. 2, 833--872.

%\bibitem[Re99]{Re99} Bertrand R\'emy, Construction de r\'eseaux en th\'eorie de Kac-Moody, C. R. Acad. Sci. Paris S\'er. I Math. 329 (1999), no. 6, 475--478 

%\bibitem[Tits87]{Tits87} Jacques Tits, Uniqueness and presentation of Kac-Moody groups over fields, J. Algebra 105 (1987), no. 2, 542--573. 




%\bibitem[Po08]{Po08} S. Popa. \textit{On the superrigidity of malleable actions with spectral gap}. J. Amer. Math. Soc.  21  (2008),  no. 4, 981--1000.


\bibitem[Th10]{Th10} Andreas Thom,  Examples of hyperlinear groups without factorization property. Groups Geom. Dyn. 4 (2010), no. 1, 195--208.

\bibitem[Va08]{Va08} Stefaan Vaes, Explicit computations of all finite index bimodules for a family of $II_1$ factors. Annales Scientifiques de l'Ecole Normale Sup\'erieure 41 (2008), 743--788.

\bibitem[Zi76]{Zi76} Robert J. Zimmer, Extensions of ergodic group actions. Illinois J. Math. 20 (1976), no. 3, 373--409.

\end{thebibliography}
}
\end{document}